\documentclass[12pt,a4paper]{article}

\usepackage{a4wide}

\usepackage[latin1]{inputenc}
\usepackage[T1]{fontenc}
\usepackage[english]{babel}

\usepackage{amsmath}
\usepackage{amssymb}
\usepackage{amsfonts,makeidx,amsthm,mathrsfs}

\newcommand{\Lr}{\mathscr{L}}
\newcommand{\dvol}{\frac{\omega^m}{m!}}

\newcommand{\dsubvol}{\frac{\omega^{m-1}}{(m-1)!}}

\def\End{\mathop{\mathrm{End}}\nolimits}
\newcommand{\tr}{\mathrm{tr}}

\newcommand{\hatc}{\hat{c}}
\newcommand{\hatb}{\hat{b}}
\newcommand{\hata}{\hat{a}}
\newcommand{\g}{\mathfrak{g}}

\newcommand{\R}{\mathbb{R}}

\def\ham{\mathop{\mathrm{Ham}}\nolimits}

\def\diff{\mathop{\mathrm{Diff}}\nolimits}
\def\VM{\mathop{\mathfrak{X}(M)}\nolimits}

\newcommand{\deltaFed}{\delta_{\mathrm{F}}}
\newcommand{\ric}{\mathrm{Ric}}

\newcommand{\nablaLC}{\nabla^{g_J}}
\newcommand{\nablabarJ}{\overline{\nabla}^J}

\newcommand{\V}{\mathcal{V}}

\newcommand{\J}{\mathcal{J}}

\newcommand{\D}{\mathcal{D}}
\newcommand{\RR}{\mathcal{R}}
\newcommand{\Rr}{\mathrm{R}}

\newcommand{\WW}{\mathbb{W}}
\newcommand{\ddt}{\frac{d}{dt}}
\newcommand{\dds}{\frac{d}{ds}}
\newcommand{\ddto}{\left.\ddt\right|_{0}}
\newcommand{\ddso}{\left.\dds\right|_{0}}
\newcommand{\Om}{\Omega^{\J}}
\newcommand{\Omtilde}{\widetilde{\Omega}^{\J}}

\newcommand{\invnu}{\frac{1}{\nu}}

\newcommand{\W}{\mathcal{W}}

\newtheorem{theoremprinc}{Theorem}
\newtheorem{theorem}{Theorem}[section]

\newtheorem{lemma}[theorem]{Lemma}
\newtheorem{cor}[theorem]{Corollary}
\newtheorem{prop}[theorem]{Proposition}
\newtheorem{defi}[theorem]{Definition}

\theoremstyle{definition}

\theoremstyle{remark}

\newtheorem{rem}[theorem]{Remark}

\begin{document}

\renewcommand{\refname}{Bibliography}

\title{The scalar curvature in formal deformation quantization. I}

\author{Laurent La Fuente-Gravy\\
	\scriptsize{laulafuent@gmail.com}\\
	\footnotesize{Universit\'e libre de Bruxelles and Haute-\'Ecole Bruxelles-Brabant -- ESI}\\[-7pt]
	\footnotesize{Belgium} \\[-7pt]}

\maketitle

\begin{abstract}
In the framework of formal deformation quantization, we apply our formal moment map construction on the space of almost complex structures to recover the Donaldson-Fujiki moment map picture of the Hermitian scalar curvature. In the integrable case, it yields a formal moment map deforming the scalar curvature moment map.
\end{abstract}

\noindent {\footnotesize {\bf Keywords:} Almost-complex structures, K\"ahler geometry, Moment map, Deformation quantization, Hamiltonian diffeomorphisms, diffeomorphisms group, Hermitian scalar curvature.}

\tableofcontents


\section{Introduction}

The Donaldson-Fujiki moment map picture \cite{Don2,Fuj} states the Hermitian scalar curvature is a moment map on the space $\J(M,\omega)$ of positive almost-complex structures on a symplectic manifold $(M,\omega)$. This famous picture motivates the use of GIT stability to treat the constant scalar curvature K\"ahler metric problem.

In our approach of formal moment maps \cite{LLFlast,LLFformalDonaldson}, we propose a general picture using formal deformation quantization \cite{BFFLS} to recover and deform moment map pictures on infinite dimensional spaces. This paper proposes to apply this procedure to the space $\J(M,\omega)$.

A natural Fedosov star product algebra bundle is defined above $\J(M,\omega)$. Using a canonical formal connection \cite{AMS} on that bundle, we show the star product trace of its curvature is a deformation of the symplectic form involved in the Donaldson-Fujiki picture. 

Considering the action of Hamiltonian diffeomorphisms on $\J(M,\omega)$, we show this action preserves the deformed symplectic form. In the almost-K\"ahler situation, we show the star product trace satisfies the formal moment map equation at order $1$ in $\nu$, and we show it co\"incides with the Donaldson-Fujiki picture. In the K\"ahler case, we show the star product trace of a deformed Hamiltonian gives a formal moment map on $\J(M,\omega)$ which deforms the scalar curvature.

An alternative approach was proposed by Foth-Uribe \cite{FU} using the operators from geometric quantization.


\section{Three connections to play with} \label{subsect:connections}

Throughout this paper, we consider a closed symplectic manifold $(M,\omega)$ of dimension $2m$. We also deal with infinite dimensional manifolds and Lie groups, we will follow the theory from \cite{KHN}.

We consider the space of almost-complex structures on $(M,\omega)$:
\begin{equation*}
\J(M,\omega):=\{ J\in \Gamma\End(TM)\, |\, J^2=-Id,\, \omega(J\cdot,J\cdot)=\omega(\cdot,\cdot),\, \omega(\cdot,J\cdot)>0 \}
\end{equation*}
It is a Fr\'echet manifold. At any point $J\in \J(M,\omega)$, its tangent space is 
\begin{equation*}
T_J \J(M,\omega):=\{A\in  \Gamma\End(TM)\, | \, \omega(\cdot,A\cdot) \textrm{ is symmetric and } AJ=-JA\}
\end{equation*}
The symplectic form on $\J(M,\omega)$ we will be interested in writes as :
\begin{equation} \label{eq:OmJ}
\Om_J(A,B):=\int_M\textrm{Tr}(JAB)\dvol, \textrm{ for any } A,B\in T_J\J(M,\omega).
\end{equation}
Also, $\J(M,\omega)$ admits a complex structure compatible with $\Om$
$$\mathbb{J} A := JA \textrm{ for } J\in T_J\J(M,\omega).$$
When, there is an integrable $J_0 \in \J(M,\omega)$ turning $(M,\omega,J_0)$ into a K\"ahler manifold, the subspace of integrable complex structures $\J_{int}(M,\Omega)\subseteq \J(M,\omega)$ is a complex subspace so that $\Om$ restricts to a symplectic structure
$$\Omega^{\J_{int}}:= \left.\Om\right|_{\J_{int}}.$$

To any $J\in \J(M,\omega)$, one attaches a Riemannian metric
$$g_J(\cdot,\cdot):=\omega(\cdot,J\cdot).$$
Then, one can consider three connections :
\begin{itemize}
\item the Levi-Civita connection $\nabla^{g_J}$, 
\item a symplectic connection $\nabla^J$ build out of $\nabla^{g_J}$ as in \cite{LLFformalDonaldson},
\item the Chern connection, we will denote by $\overline{\nabla}^J$.
\end{itemize}

\subsection{The Levi-Civita connection $\nabla^{g_J}$}

The Levi-Civita connection $\nabla^{g_J}$ is the unique torsion-free connection leaving $g_J$ parallel.

For $\varphi\in \ham(M,\omega)$ a Hamiltonian diffeomorphism, one has a natural action of it on $J\in \J(M,\Omega)$ by
$$\varphi \cdot J := \varphi_*\circ J \circ \varphi_*^{-1}.$$
One also has a natural action on a linear connection $\nabla$ on $TM$ by
$$(\varphi\cdot \nabla)_X Y := \varphi_*\nabla_{\varphi_*^{-1}X}\varphi_*^{-1}Y \textrm{ for all } X,Y \in \VM.$$
The next proposition follows from straightforward computations.

\begin{prop} \label{prop:equivnablagJ}
For $J\in  \J(M,\Omega)$ and $\varphi\in \ham(M,\omega)$,
$$\nabla^{g_{\varphi.J}}=\varphi.\nabla^{g_J}.$$
\end{prop}

\noindent Later, we will need a formula for the first order variation of $\nabla^{g_J}$.

\begin{lemma} \label{lemme:firstvarnabla}
Let $t\mapsto J_t\in \J(M,\omega)$ with $\ddto J_t=A$, then 
$$g_J(\ddto \nabla^{g_{J_t}}_X Y, Z) = \frac{1}{2}\left( (\nabla^{g_J}_Y a)(X,Z)+ (\nabla^{g_J}_X a)(Y,Z)- (\nabla^{g_J}_Z a)(X,Y)\right),$$
for $X,Y,Z \in \VM$ and $a(X,Y):=\ddto g_{J_t}(X,Y)=\omega(X,AY)$.
\end{lemma}

\begin{proof}
A short proof can be found in P. Topping's book \cite{PTbook}.
\end{proof}

The \emph{curvature} of $\nabla^{g_J}$ is the tensor:
$$R^{g_J}(X,Y)Z:= \nabla_X \nabla_Y Z - \nabla_Y \nabla_X Z - \nabla_{[X,Y]} Z$$
for $X,Y,Z \in \VM$. The \emph{Ricci curvature} is the symmetric $2$-tensor 
$$\ric(U,V):= \sum_{k=1}^{2m} g_J(R^{g_J}(e_k,U)V,e_k),$$
for $U,V \in T_xM$ and $\{e_k\,|\, k=1,\ldots, 2m\}$ is an orthonormal frame at point $x\in M$.\\
In the K\"ahler case, when $J$ is integrable : $\ric(JU,JV)=\ric(U,V)$ and one defines the Ricci form by
\begin{equation}\label{eq:ricciform}
\mathrm{ric}(U,V):= \ric(JU,V).
\end{equation}
Also, one has 
\begin{equation} \label{eq:RicciKahler}
\mathrm{ric}(U,V)=-\frac{1}{2}\sum_{k=1}^{2m} g_J(R^{g_J}(e_k,Je_k)U,V).
\end{equation}

\subsection{The symplectic connection $\nabla^J$} \label{sect:nablaJ}

In the sequel, we will need to attach a symplectic connection to any almost complex structure $J\in \J(M,\omega)$. It is similar to what we used in \cite{LLFformalDonaldson}.

First, recall that a symplectic connection on $(M,\omega)$ is a torsion-free linear connection leaving $\omega$ parallel. A symplectic connection can be build out of any torsion-free linear connection, so we build one out of $\nabla^{g_J}$, for any $J\in \J(M,\omega)$.

We define a $2$-tensor $K^J(X,Y)$ on $M$ by
\begin{equation*} \label{eq:tensorSJ}
\omega(K^J(X,Y),Z):= (\nabla^{g_J}_X\omega)(Y,Z) \textrm{ for all } X,Y,Z \in TM. 
\end{equation*}
Then, a symplectic connection $\nabla^J$ is obtained through the formula
\begin{equation*}
\nabla^J_X Y:= \nabla^{g_J}_X Y + \frac{1}{3}K^J(X,Y) + \frac{1}{3}K^J(Y,X),
\end{equation*}
for $X,Y \in \VM$.

\begin{prop}\label{prop:nablaJ}
For all $X,Y\in TM$, one has 
$$K^J(X,Y)=-J\left(\nabla^{g_J}_X J\right)(Y).$$
So that,
$$\nabla^J_X Y = \nabla^{g_J}_X Y - \frac{1}{3}J\left(\nabla^{g_J}_X J\right)(Y) - \frac{1}{3}J\left(\nabla^{g_J}_YJ\right)(X).$$

\noindent Moreover, $\nabla^{\varphi\cdot J}=\varphi\cdot \nabla^J$ for any $\varphi\in \ham(M,\omega)$.
\end{prop}

\begin{proof}
The formula for $K^J$ follows from
\begin{eqnarray*}
(\nabla^{g_J}_X \omega)(Y,Z) & = & (\nabla^{g_J}_X g_J(J\cdot,\cdot))(Y,Z), \\
 & = & g_J(\left(\nabla^{g_J}_X J\right)Y,Z),\\
 & = & \omega(-J\left(\nabla^{g_J}_X J\right)(Y),Z).
\end{eqnarray*}

\noindent The equivariance with respect to the action of $\varphi$ is a consequence of Proposition \ref{prop:equivnablagJ}.
\end{proof}

\subsection{The Chern connection $\nablabarJ$ and the Hermitian scalar curvature}

For any $J\in \J(M,\omega)$, the Chern connection $\nablabarJ$ is a canonical $J$-linear connection on the complex vector bundle $(TM,J)$. It is defined by
$$\nablabarJ_X Y:= \nabla^{g_J}_X Y - \frac{1}{2} J\left(\nabla^{g_J}_X J\right)(Y),$$ 
for any $X,Y\in \VM$. The Chern connection $\nablabarJ$ preserves $g_J,J$ and then $\omega$, but has torsion. It also preserves the Hermitian metric 
$$h^J(X,Y):=g_J(X,Y)-i\omega(X,Y), \textrm{ for any } X,Y\in TM$$
which is $J$-linear in the first entry and $J$-anti-linear in the second one.

\begin{prop}
$\overline{\nabla}^{\varphi\cdot J}=\varphi\cdot \nablabarJ$ for any $\varphi\in \ham(M,\omega)$ and $J\in \J(M,\omega)$.
\end{prop}

\noindent The proof again follows from Proposition \ref{prop:equivnablagJ}.

Let us now introduce the main characters of this paper: the Hermitian Ricci form and the Hermitian scalar curvature.

The connection $\nablabarJ$ induces a connection on the complex line bundle $\Lambda^m (TM,J)$ still denoted $\nablabarJ$. Consider a local complex basis $\mathcal{Z}:=\{Z_1,\ldots,Z_m\}$ of $(TM,J)$. This basis induces a non-zero local section $\zeta:=Z_1\wedge \ldots \wedge Z_m$ of $(TM,J)$. One computes
\begin{equation*}
\nablabarJ_X \zeta:= \theta^J_{\mathcal{Z}}(X)\zeta
\end{equation*}
for $\theta^J_{\mathcal{Z}}(X):=(h^J)^{ki}h^J(\nablabarJ_X Z_i, Z_k)$, with $(h^J)^{ki}$ denoting the inverse of the matrix of $h^J$ in the basis $\mathcal{Z}$, and we use from now on the summation convention on repeated indices. The $1$-form $\theta^J_{\mathcal{Z}}$ is only locally defined, it depends on the choice of the local complex basis $\mathcal{Z}$ but its differential is globally defined.

The \emph{Hermitian Ricci form} of $(M,\omega,J)$ is the real form 
\begin{equation*}\label{def:HRicciform}
\rho^J:=i\,d\theta^J_{\mathcal{Z}}.
\end{equation*}
In the K\"ahler case, it co\"incides with the Ricci form in Equation \eqref{eq:ricciform}, but not in general.
The \emph{Hermitian scalar curvature} is the function $S^J$ such that :
\begin{equation*}  \label{def:Hscalarcurv}
\rho^J\wedge \dsubvol=\frac{1}{2} S^J \dvol,
\end{equation*}
or $S^J:=-\Lambda^{ql}\rho^J_{ql}$, for $\Lambda$ being the inverse matrix of the (real) coordinate matrix of $\omega$.

We will need the first order variation of $\rho^J$ which writes in term of the first order variation of $\nablabarJ$. Actually, varying $J$ makes the complex structure on $(TM,J)$ vary. So we need to compensate that, we follow the ideas from the book \cite{Gaud}. 

We consider a path of almost complex structures 
$$J_t:=\gamma_t \circ J \circ \gamma_t^{-1}$$
for $\gamma_t:=\exp(ta)$ and $a=\frac{1}{2} JA$ for $A\in T_J \J(M,\omega)$, so that $\ddto J_t=A$. We consider the path of $J$-linear connections
$$\widetilde{\nabla}^t:=\gamma_t^{-1}\circ \overline{\nabla}^{J_t} \circ \gamma_t.$$
Consider the local unitary complex basis $\mathcal{Z}_t:=\{\gamma_t Z_1, \ldots, \gamma_t Z_m\}$ of $(TM,J_t)$ starting from a chosen local unitary complex basis $\mathcal{Z}:=\{ Z_1, \ldots, Z_m\}$ of $(TM,J)$. One build the local non zero section $\zeta_t:=\gamma_t Z_1 \wedge\ldots\wedge \gamma_t Z_m$. Then, for any $X\in TM$,
$$\overline{\nabla}^{J_t}_X \zeta_t = \theta^{J_t}_{\mathcal{Z}_t}(X) \zeta_t.$$
On the other hand,
$$\widetilde{\nabla}^t_X\zeta =  \theta^{J_t}_{\mathcal{Z}_t}(X) \zeta.$$
which means the $1$-form $\kappa:=\ddto \theta^{J_t}_{\mathcal{Z}_t}$ is globally defined. Moreover,
$$\ddto \rho^{J_t}=i\,d\kappa.$$
The $1$-form $\kappa$ is called the \emph{first order variation of the Chern connection}.

\begin{prop}
For $X\in TM$, 
\begin{equation*}
\kappa(X)=\frac{i}{2}\, \delta^J A^{b}(X), 
\end{equation*}
where $\ddto J_t=A$, $Y^b$ is the $1$-form $g_J(Y,\cdot)$ and $\delta^J T(X_1,\ldots,X_n) := -(\nabla^{g_J}_{e_i}T)(e_i,X_1,\ldots, X_n)$ for $T$ a $n$-tensor on $M$, $\{e_i\,|\,i=1,\ldots,2m\}$ a $g_J$-orthonormal frame and $X_1,\ldots,X_n \in TM$.
\end{prop}

\noindent For a proof of the above Proposition, see Gauduchon's book \cite{Gaud}.

\begin{cor} \label{cor:variationHermitian}
For a path $t\mapsto J_t\in \J(M,\omega)$, with $\ddto J_t=A$, one computes
$$\ddto \rho^{J_t}=-\frac{1}{2} d\delta^J A^{b} \textrm{ and } \ddto S^{J_t}=\frac{1}{2}\Lambda^{ql}\left(d\delta^J A^{b}\right)_{ql}.$$
\end{cor}

\begin{rem} \label{rem:deltas}
We will keep the superscript $J$ in $\delta^J$ all along to emphasize its dependence in $J$ through $g_J$ but also to avoid confusion with the $\deltaFed$ from Fedosov construction.\\
The musical isomorphism $b$ also depends on $J$, when this dependence will be investigated we will write $b_J$.
\end{rem}

Finally, the equivariance of the Chern connection translates into the equivariance of the Hermtian Ricci form.

\begin{lemma} \label{lemme:equivHRicci}
For $J\in \J(M,\omega)$ and $\varphi\in \ham(M,\omega)$,
$$\rho^{\varphi^{-1}.J}=\varphi^*\rho^J.$$
\end{lemma}

\section{Formal connections and curvature}

\subsection{Fedosov construction}

On $(M,\omega)$, consider a basis $\{e_1,\ldots,e_{2n}\}$ of $T_xM$ at $x\in M$ and its dual basis $\{y^1,\ldots, y^{2n}\}$ of $T^*_xM$. The algebra of formal symmetric forms on $T_xM$ of the kind:
$$a(y,\nu):=\sum_{2r+k=0}^{\infty} \nu^k a_{k,i_1\ldots i_r}y^{i_1}\ldots y^{i_r},$$
where $a_{k,i_1\ldots i_r}$ symmetric in $i_1\ldots i_r$ and $2k+r$ is the total degree, with product, 
\begin{eqnarray*}
(a\circ b)(y,\nu) & := &\left.\left( \exp\left(\frac{\nu}{2}\Lambda^{ij} \partial_{y^i} \partial_{z^j}\right)a(y,\nu)b(z,\nu)\right)\right|_{y=z}, \nonumber \\
\end{eqnarray*}
for two formal symmetric tensors $a(y,\nu)$ and  $b(y,\nu)$, is called the formal Weyl algebra $\WW_x$.

The formal Weyl algebra bundle is the bundle $\W:=\bigsqcup_{x\in M}\WW_x$ over $M$. Denote by $\Gamma\W\otimes \Lambda M$ the space of differential forms with values in sections of $\W$. Such a differential form writes locally as:
\begin{equation} \label{eq:sectionofW}
\sum_{2k+l\geq 0,\, k,l\geq 0, p\geq 0} \nu^k a_{k,i_1\ldots i_l,j_1\ldots j_p}(x)y^{i_1}\ldots y^{i_l}dx^{j_1}\wedge \ldots \wedge dx^{j_p},
\end{equation}
with $a_{k,i_1\ldots i_l,j_1\ldots j_p}(x)$ are symmetric in the $i$'s and antisymmetric in the $j$'s. The space $\Gamma \W\otimes \Lambda^*M$ is filtered with respect to the total degree
$$\Gamma \W\otimes \Lambda^*M \supset \Gamma \W^1\otimes \Lambda^*M\supset \Gamma \W^2\otimes \Lambda^*M \supset \ldots.$$
The $\circ$-product extends fiberwisely to $\Gamma \W\otimes \Lambda^*M$ making it an algebra. That is, for $a, b\in \Gamma \W$ and $\alpha, \beta\in \Omega^*(M)$, we define
$(a\otimes \alpha) \circ (b\otimes \beta) := a\circ b \otimes \alpha\wedge \beta$. It is a graded Lie algebra for the graded commutator $[s,s']:=s\circ s'- (-1)^{q_1q_2}s'\circ s$ where $s$, resp. $s'$ are of anti-symmetric degree $q_1$, resp. $q_2$ makes $\W$-valued forms.

From a symplectic connection $\nabla$ on $(M,\omega)$, one defines a derivation $\partial$ of anti-symmetric degree $+1$ on $\W$-valued forms by :
\begin{equation*}
\partial a := da + \frac{1}{\nu}[\overline{\Gamma},a]  \textrm{ for } a\in \Gamma \W\otimes \Lambda M,
\end{equation*}
where $\overline{\Gamma}:=\frac{1}{2}\omega_{lk}\Gamma^k_{ij}y^ly^jdx^i$, for $\Gamma^k_{ij}$ the Christoffel symbols of $\nabla$ on a Darboux chart. 

Setting $\overline{R}:= \frac{1}{4} \omega_{ir}R^r_{jkl}y^iy^jdx^k\wedge dx^l$, for $R^r_{jkl}:=\left(R(\partial_k,\partial_l)\partial_j\right)^r$ the components of the curvature tensor of $\nabla$, the curvature of $\partial$ is
\begin{equation*}
\partial\circ \partial\, a := \frac{1}{\nu}[\overline{R},a].
\end{equation*}

We look for flat connections on $\Gamma \W$ of the form
\begin{equation*} \label{eq:defD}
D a:=\partial a - \deltaFed a + \frac{1}{\nu}[r,a],
\end{equation*}
for $r$ a $\W$-valued $1$-form and $\deltaFed$ is defined by
\begin{equation*} \label{eq:deltadef}
\deltaFed(a) := dx_k\wedge \partial_{y_k} a=-\frac{1}{\nu}[\omega_{ij}y^i dx^j,a],
\end{equation*}
the F subscript is there to avoid confusion with $\delta^J$, see remark \ref{rem:deltas}.
The curvature of $D$ is
\begin{equation*}
D^2 a = \frac{1}{\nu}\left[\overline{R} + \partial r - \deltaFed r + \frac{1}{2\nu}[r,r]-\omega,a\right].
\end{equation*}

Define 
$$\deltaFed^{-1} a_{pq}:= \frac{1}{p+q}y^ki(\partial_{x^k})a_{pq} \textrm{ if } p+q>0 \textrm{ and } \deltaFed^{-1}a_{00}=0,$$
where $a_{pq}$ is a $q$-form with $p$ $y$'s and $p+q>0$. For any given closed central $2$-form $\Omega$,
there exists a unique solution $r \in \Gamma \W \otimes \Omega^1 M$ with $\W$-degree at least $3$ of equation:
\begin{equation*}\label{eq:req}
\overline{R} + \partial r - \deltaFed r + \frac{1}{\nu}r\circ r = \Omega,
\end{equation*}
and satisfying $\deltaFed^{-1}r=0$, see Fedosov \cite{fed2}. Because $\Omega$ is central for the $\circ$-product, it makes $D$ flat.

To the flat connection $D$, one attaches the space of flat sections $\Gamma \W_{D} := \{a\in \Gamma \W | D a=0\}$. Flat sections form an algebra for the $\circ$-product as $D$ is a derivation. The symbol map is defined by $\sigma :a\in \Gamma \W_{D} \mapsto \left.a\right|_{y=0}\in C^{\infty}(M)[[\nu]]$. The map $\sigma$ is a bijection with inverse $Q$ (Fedosov \cite{fed2}) defined by 
\begin{equation*} \label{eq:defQ}
Q:=\sum_{k\geq 0} \left(\deltaFed^{-1}(\partial + \frac{1}{\nu}[r,\cdot])\right)^k.
\end{equation*}

The \emph{Fedosov star product} $*$ build with the data of $\Omega$ a formal closed $2$-form and $\nabla$ a symplectic connection is, for all $F,G\in C^{\infty}(M)[[\nu]]$:
$$F*G:= \left.\left(Q(F)\circ Q(G)\right)\right|_{y=0}.$$

\begin{defi}
To $J\in \J(M,\omega)$, we attach the \emph{star product $*_J$} which is the Fedosov star product build with $\Omega=\nu\rho^J$ and symplectic connection $\nabla^J$.
\end{defi}

In the sequel, when dealing with the star product $*_J$, we may emphasize the dependence in $J$ by writing $\overline{\Gamma}^J,r^J, D^J, Q^J,\ldots $ for the corresponding ingredients of the Fedosov construction performed with the symplectic connection $\nabla^J$ and $\Omega=\nu\rho^J$.


\subsection{The star products $\{*_J\}_{J\in \J(M,\omega)}$ and a formal connection}

\begin{defi}
Define the star product algebra bundle $\mathcal{V}$ over $\J(M,\omega)$ by 
$$\mathcal{V}:= \J(M,\omega) \times C^{\infty}(M)[[\nu]] \stackrel{p}{\rightarrow} \J(M,\omega),$$ 
where the fiber $J\in \J(M,\omega)$ is equipped with the star product $*_{J}$ and $p$ is the projection. 
\end{defi}

\noindent A \emph{formal connection} $\D$ on sections of $\V$ is an operator of the form 
$$d^{\J}+\beta,$$
with a formal series $\beta=\sum_{k\geq 1} \nu^k \beta_k$ of $1$-forms on $\J(M\omega)$ with values in differential operators on functions of $M$. We say the connection is \emph{compatible} with the family of star products $\{*_{J}\}_{J\in \J(M,\omega)}$ when, for all sections $F,G$ of $\V$:
$$\D(F*_{J}G)=\D(F)*_{J} G + F*_{J} \D(G)$$

Such a connection exists \cite{AMS}. To define it one needs two technical lemmas about Fedosov construction of star products.

\begin{lemma}\label{lemme:Dinverse}
Suppose $b \in \Gamma \W\otimes \Lambda^1M$ satisfies $Db = 0$. Then the equation $Da = b$ admits a
unique solution $a \in \Gamma \W$, such that $a|_{y=0} = 0$, it is given by
$$b=D^{-1}a:=-Q(\deltaFed^{-1}a).$$
\end{lemma}


As in \cite{LLFlast,LLFformalDonaldson}, we make use of a canonical lift of smooth path on the base manifold to isomorphisms of Fedosov star products algebra. 

Define sections of the extended bundle $\W^+\supset \W$ as locally of the form
\begin{equation*} \label{eq:sectionofW+}
\sum_{2k+l\geq 0, l\geq 0} \nu^k a_{k,i_1\ldots i_l}(x)y^{i_1}\ldots y^{i_l}.
\end{equation*}
similar to \eqref{eq:sectionofW}, with $p=0$, but we allow $k$ to take negative values, the total degree $2k+l$ of any term must remain nonnegative and in each given nonnegative total degree there is a finite number of terms.

Given a smooth path $t\mapsto J_t \in \J(M,\omega)$ with $\ddt J_t := A_t$. Hence, by Corollary \ref{cor:variationHermitian}, $\ddt \rho^{J_t}=-\frac{1}{2}d(\delta^{J_t} A_t^{b_{J_t}})$. The following Theorem comes from \cite{fed2} and is adapted to the particular case of Fedosov star products of the form of $*_{J}$.

\begin{theorem}\label{theor:smoothisom}
Consider smooth paths $t\in [0,1] \mapsto J_t\in \J(M,\omega)$. 
Then there exists maps $B_t:\Gamma\W \rightarrow \Gamma \W$ defined by
$$B_t a:=v_t\circ a \circ v_t^{-1}$$
for $v_t \in \Gamma \W^+$ being the unique solution of the initial value problem:
\begin{equation*} \label{eq:vt}
\left\{\begin{array}{rcl}
\frac{d}{dt}v_t & = & \frac{1}{\nu}h_t\circ v_t \\
v_0 &= & 1
\end{array}\right.
\end{equation*}
with
\begin{equation*} \label{eq:ht}
h_t:=-(D^{J_t})^{-1}\left(\ddt \overline{\Gamma}^{J_t}+\ddt r^{J_t} + \frac{\nu}{2} \delta^{J_t} A_t^{b_{J_t}}\right).
\end{equation*}
Moreover, $B_t(D^{J_0} a)= D^{J_t}(B_t a)$ for all $a\in \Gamma\W$ so that $$\left.B_t\right|_{\Gamma\W_{D^{J_0}}}:\Gamma\W_{D^{J_0}} \rightarrow \Gamma \W_{D^{J_t}}$$ is an isomorphism of flat sections algebras and hence
\begin{equation} \label{eq:equivparall}
\sigma\circ B_t \circ Q^{J_0}:(C^{\infty}(M)[[\nu]],*_{J_0}) \rightarrow (C^{\infty}(M)[[\nu]],*_{J_t})
\end{equation}
is an equivalence of star product algebras.
\end{theorem}

The dependence of $h_t$ in $J_t$ and its covariant derivatives is polynomial which makes the paths $t\mapsto h_t$  and $t\mapsto v_t$ smooth.

Following \cite{AMS}, one defines a compatible formal connection $\D$ by interpreting the above Theorem as a parallel lift of the path $t\mapsto J_t$.

\begin{defi} \label{def:formalconn}
For $A\in T_J\J(M,\omega)$, with $t\mapsto J_t$ so that $\ddto J_t=A$, define :
\begin{itemize}
\item the \emph{connection $1$-form} $\alpha\in \Omega^1(\J(M,\omega),\Gamma\W^3)$ by
\begin{equation*} \label{eq:defalpha}
\alpha_{J}(A):=(D^{J})^{-1}\left(\left.\frac{d}{dt}\right|_0\overline{\Gamma}^{J_t}+ \ddto r^{J_t} + \frac{\nu}{2} \delta^J A_t^{b_J}\right),
\end{equation*}
\item the $1$-form $\beta$ with values in formal differential operators:
\begin{equation*}
\beta_{J}(A)(F):= \left.\invnu[\alpha_{J}(A),Q^{J}(F)]\right|_{y=0},\,\textrm{ for }F\in C^{\infty}(M)[[\nu]],
\end{equation*}
\item the \emph{formal connection} $\D:= d^{\J}+\beta$.
\end{itemize}
\end{defi}

\begin{prop}
$\D$ is a formal connection on $\V$ compatible with the family of Fedosov star products $\{*_{J}\}_{J\in \J(M,\omega)}$.
Moreover, the parallel transport for $\D$ along the path $t\mapsto J_t\in \J(M,\omega)$ is given by the equivalence of star product algebra obtained from Theorem \ref{theor:smoothisom}.
\end{prop}

\noindent The compatibility of $\D$ is proved in \cite{AMS} and the link with parallel transport can be proved similarly to the corresponding statement in \cite{LLFformalDonaldson}.

\subsection{The curvature of $\D$}

The curvature of $\D$ evaluated at vector fields $X,Y$ on $\J$ acting on a section $F$ of $\V$ is:
$$\left(\RR(Y,Z)F\right)(J):=\left(\D_{Y}(\D_{Z} F) - \D_{Z}(\D_{Y} F) - \D_{[Y,Z]} F\right)(J),$$

To properly compute the curvature tensor on vectors $A,B\in T_J \J(M,\omega)$, we use extensions of $A$ and $B$ as vector fields on $\J(M,\omega)$.

For any $a\in \End(TM,\omega)$, one defines a vector field $\hat{a}$ by
\begin{equation*} \label{eq:ahat}
\hat{a}_{\widetilde{J}}:= \ddto \exp(ta)\circ \widetilde{J} \circ  \exp(-ta) \in T_{\widetilde{J}} \J(M,\omega).
\end{equation*}
In such a way, for $a:=\frac{1}{2} JA$ with $A\in T_J\J(M,\omega)$ , one has an extension of $A$ as $\hat{a}_J=A$.

One also computes the Lie bracket of two such vector fields obtained from $a,b\in \End(TM,\omega)$:
\begin{equation*}
[\hat{a},\hat{b}]_{\widetilde{J}}:=-\widehat{[a,b]}_{\widetilde{J}}
\end{equation*}
In the particular case of $a:=\frac{1}{2} JA$ and $b:=\frac{1}{2} JB$ with $A,B\in T_J\J(M,\omega)$, the above Lie bracket evaluated at $J$ vanishes:
\begin{equation*}
[\hat{a},\hat{b}]_{J}=0.
\end{equation*}

Hence, we have the following formula for the curvature on $A,B\in T_J\J(M,\omega)$ acting on a section $F$ of $\J$:
$$(\RR(A,B)F)(J):=\left(\D_{\hat{a}}(\D_{\hat{b}} F) - \D_{\hat{b}}(\D_{\hat{a}} F)\right)(J),$$
using the natural extensions $\hat{a},\hat{b}$ defined above with $a=\frac{1}{2} JA$ and $b:=\frac{1}{2} JB$.

\begin{theorem} \label{theor:Ralpha}
For $A,B\in T_J\J(M,\omega)$ and a section $F$ of $\V$, the curvature of $\D$ is given by
\begin{equation}\label{eq:Rdef}
(\RR(A,B)F)(J)=\left.\invnu[\Rr_J(A,B),Q^J(F(J))]\right|_{y=0}
\end{equation}
for $\Rr_J(A,B)$ being the $2$-form with values in $\Gamma\W$ defined by
\begin{equation}\label{eq:Ralpha}
\Rr_J(A,B):=\frac{\nu}{4}\mathrm{Tr }(JAB)  +  d^{\J}\alpha_{J}(A,B)+\invnu[\alpha_{J}(A),\alpha_{J}(B)],
\end{equation}
Moreover, 
\begin{itemize}
\item $\Rr_{J}(A,B)\in \Gamma \W_{D^{J}}$,
\item $\left.\Rr_{J}(A,B)\right|_{y=0}=\frac{\nu}{4}\mathrm{Tr}(JAB)+ O(\nu^2)$.
\end{itemize}
\end{theorem}

\begin{proof}
The terms containing $\alpha$ in Equation \eqref{eq:Ralpha} come from standard computations of $\RR$. The term in $\nu$ from Equation \eqref{eq:Ralpha} doesn't contribute in Equation \eqref{eq:Rdef} but will make $\Rr_J(A,B)$ a flat section.

To check  $\Rr_{J}(A,B)\in \Gamma \W_{D^{J}}$, we compute $D^{J}$ applied to the RHS of \eqref{eq:Ralpha}. First, because $\frac{\nu}{4}\mathrm{Tr }(JAB)$ is a function on $M$, 
$$D^{J}  \mathrm{Tr }(JAB)=d(\mathrm{Tr }(JAB)).$$
Now, we detail the terms of $D^{J}\left(d^{\J}\alpha_{J}(A,B)+\invnu[\alpha_{J}(A),\alpha_{J}(B)]\right)$. To do that we use the extensions $\hat{a}$ and $\hat{b}$ defined earlier for $a=\frac{1}{2} JA$ and $b:=\frac{1}{2} JB$ and we start with $D^J(\hata(\alpha(\hatb)))$. Consider the 2-parameter family of almost complex structures
$$J_{st}:= \exp(sb)\exp(ta)J\exp(-ta)\exp(-sb),$$
so that 
$$\ddso J_{st}=[b,\exp(ta)J\exp(-ta)].$$
We compute
\begin{eqnarray}
D^J(\hata(\alpha(\hatb))) & = & D^J(\ddto \alpha_{J_{0t}}(\ddso J_{st})), \\
 & = & \ddto D^{J_{0t}} \alpha_{J_{0t}}(\ddso J_{st}) - \left(\ddto D^{J_{0t}}\right)(\alpha_{J}(\ddso J_{s0})). 
\end{eqnarray}
Similarly, for $D^J(\hatb(\alpha(\hata)))$, consider the 2-parameter family of almost complex structures
$$\widetilde{J}_{st}:= \exp(sa)\exp(tb)J\exp(-tb)\exp(-sa),$$
so that 
$$\ddso \widetilde{J}_{st}=[a,\exp(tb)J\exp(-tb)].$$
Then,
\begin{eqnarray}
D^J(\hatb(\alpha(\hata))) & = & D^J(\ddto \alpha_{\widetilde{J}_{0t}}(\ddso \widetilde{J}_{st})), \\
 & = & \ddto D^{\widetilde{J}_{0t}} \alpha_{\widetilde{J}_{0t}}(\ddso \widetilde{J}_{st}) - \left(\ddto D^{\widetilde{J}_{0t}}\right)(\alpha_{\widetilde{J}}(\ddso \widetilde{J}_{s0})).
\end{eqnarray}
We have no contribution from $[\hata,\hatb]_J$ at the point $J$.

It follows from standard computation, as in \cite{LLFlast,LLFformalDonaldson}, that in 
$$D^J\left(d^{\J}\alpha_{J}(\hata,\hatb)+\invnu[\alpha_{J}(\hata_J),\alpha_{J}(\hatb_J)]\right)$$
all the terms involving $\overline{\Gamma}^J, r^J$ and $\alpha_J$ cancel with each other. So that all it remains is
\begin{equation}
D^J\left(d^{\J}\alpha_{J}(\hata,\hatb)+\invnu[\alpha_{J}(\hata_J),\alpha_{J}(\hatb_J)]\right)= \frac{\nu}{2}\left( \ddto\left[\delta^{J_{0t}}(B)\right]^{b_{J_{0t}}}- \ddto\left[\delta^{\widetilde{J}_{0t}}(A)\right]^{b_{\widetilde{J}_{0t}}}\right),
\end{equation}
and again we have no contribution from $[\hata,\hatb]_J$ at point $J$.

Using Lemma \ref{lemme:delta}, we get
\begin{eqnarray}
D^J\left(d^{\J}\alpha_{J}(\hata,\hatb)+\invnu[\alpha_{J}(\hata_J),\alpha_{J}(\hatb_J)]\right) & = &  -\frac{\nu}{4} d(\mathrm{Tr}(JAB)),\nonumber
\end{eqnarray}
which shows that $R_J(A,B)$ is a $D^J$-flat section.

\noindent Finally, because $\alpha(\cdot)$ is of degree at least 3, we have
$$\left.\Rr_{J}(A,B)\right|_{y=0}= \frac{\nu}{4} \mathrm{Tr}(JAB) + O(\nu^2).$$
\end{proof}

\begin{lemma} \label{lemme:delta}
For $A,B\in T_J\J(M,\omega)$ and $J_{ts},\widetilde{J}_{ts}$ the $2$-parameters families defined in the proof of the above Theorem \ref{theor:Ralpha}, we have
$$\ddto\left[\delta^{J_{0t}}(B)\right]^{b_{J_{0t}}}- \ddto\left[\delta^{\widetilde{J}_{0t}}(A)\right]^{b_{\widetilde{J}_{0t}}}=-\frac{1}{2} d(\mathrm{Tr}(JAB)).$$
\end{lemma}

\noindent The proof is postponed to the Appendix.

\section{Formal symplectic form and formal moment map}

\subsection{A formal symplectic form on $\J(M,\omega)$}

A \emph{formal symplectic form} on a manifold $F$ is a formal deformation of a symplectic form $\sigma_0$ of the form:
$$\sigma:=\sigma_0+\nu\sigma_1+\ldots \in \Omega^2(F)[[\nu]],$$
with closed $2$-forms $\sigma_i$ for all $i$.

As in our previous works \cite{LLFlast,LLFformalDonaldson}, the $*$-product trace and the curvature element $\Rr$ will produce our formal moment map picture.

Consider a star product $*$ on a symplectic manifold, a \emph{trace} for $*$ on a symplectic manifold $(M,\omega)$ is a character
$$ \tr : (C_c^\infty(M)[[\nu]],[\cdot,\cdot]_\ast) \to \R[\nu^{-1},\nu]].$$
A trace always exits for a given star product on $(M,\omega)$. It is unique if one asks for the normalisation condition
\begin{equation*}\label{normalized_tr}
\tr(F) = \frac1{(2\pi\nu)^m} \int_M BF\ \frac{\omega^m}{m!}, \textrm{ for all } F\in C_c^\infty(U)[[\nu]].
\end{equation*}
for all $U$ contractible Darboux chart and $B$ being local equivalences of $*|_{C^\infty(U)[[\nu]]}$ with the Moyal star product $\ast_{\mathrm{Moyal}}$.
The trace is given by the $L^2$-product with a formal function $\rho\in C^{\infty}(M)[\nu^{-1},\nu]]$, called the trace density
$$\tr(F) =\frac{1}{(2\pi\nu)^m} \int_M F\rho\ \frac{\omega^m}{m!}.$$

\noindent For $J\in \J(M,\omega)$, we denote by \emph{$\tr^{*_{J}}$ the normalised trace} of the Fedosov star product $*_{J}$ and by $\rho^{J}$ its \emph{trace density}.

\begin{defi} \label{def:Omtilde}
Let $\Omtilde$ be the formal $2$-form on $\J(M,\omega)$ defined by
$$\Omtilde_J(A,B):=4(2\pi)^m\nu^{m-1} \tr^{*_{J}}(\left.\Rr_J(A,B)\right|_{y=0}),$$
for $J\in \J(M,\omega)$ and $A,B\in T_J\J(M,\omega)$.
\end{defi}

\begin{theorem}\label{theor:omegatilde}
$\Omtilde$ is a formal symplectic form on $\J(M,\omega)$ deforming $\Om$ and invariant under the action of $\ham(M,\omega)$ on $\J(M,\omega)$.
\end{theorem}

\begin{proof}
The result follows from direct adaptation of the corresponding results from \cite{LLFlast, LLFformalDonaldson}. 

The fact that $d^{\J}\Omtilde=0$ at all $J\in \J(M,\omega)$ is computed on vector fields of the form $\hata,\hatb$ and $\hatc$ extending tangent elements $A,B$ and $C$ at $J$ and using the following Lemma.

\begin{lemma}[\cite{FutLLF2}] \label{lemme:tracevariation}
Let $t\mapsto J_t$ be a smooth path in $\J(M,\omega)$. Then
$$ \ddto \tr^{*_{J_t}}(F)=\tr^{*_{J_0}}\left(\left.\invnu[\alpha_{J_0}(\ddto J_t),Q^{J_0}(F)]\right|_{y=0}\right). $$
\end{lemma}

The invariance of $\Omtilde$ with respect to the action of $\ham(M,\omega)$, comes from the naturality of Fedosov construction and the equivariance of the ingredients we used: the Hermitian Ricci form, see Lemma \ref{lemme:equivHRicci} and the symplectic connection $\nabla^J$, see Proposition \ref{prop:nablaJ}. 

Finally, to see $\Omtilde$ deforms $\Om$, notice the trace starts with a multiple of the integral, the first order term of $\Rr$ in Theorem \ref{theor:Ralpha} is precisely $\frac{\nu}{4}\mathrm{Tr}(JAB)$ and compare with Equation \eqref{eq:OmJ}.
\end{proof}

\subsection{Deforming the Donaldson-Fujiki picture}

Consider an action $\cdot$ of a regular Lie group $G$ on $(X,\sigma)$  a manifold equipped with a formal symplectic form $\sigma$ so that the action preserves $\sigma$. We define an \emph{equivariant formal moment map} to be a map 
$$\theta:\g \rightarrow C^{\infty}(X)[[\nu]],$$
for $\g$ the Lie algebra of $G$, such that for all $g\in G$, $\mathcal{Y}\in \g$ and $x\in X$
\begin{eqnarray} \label{eq:formalmoment}
\textrm{(formal moment map) } & & \imath\left(\left.\frac{d}{dt}\right|_{t=0}\exp(t\mathcal{Y})\cdot x\right)\sigma  =  d^X \theta(\mathcal{Y}) \\
\textrm{(equivariance) }& & \theta(Ad(g)\mathcal{Y}) = (g^{-1}\cdot)^*\theta(\mathcal{Y}). \nonumber
\end{eqnarray}

\begin{theoremprinc} \label{theor:formalDFpicture} \ 

\begin{enumerate}
\item The map 
$$\mu: C^{\infty}_0(M) \rightarrow C^{\infty}(\J(M,\omega))[[\nu]]:H \mapsto \left[ J\mapsto 4(2\pi)^m\nu^{m-1}\tr^{*_{J}}(H)\right],$$
satisfies the equivariant formal moment map at first order in $\nu$ for the action of $\ham(M,\omega)$ on $(\J(M,\omega),\Omtilde)$ .\\
\item In the K\"ahler case, denote by $\Delta^J H:=-\frac{1}{2}\Lambda^{ks}\left(d(d H\circ J)\right)_{ks}$ the Laplacian for $J\in \J_{int}(M,\omega)$, then the map
$$\widetilde{\mu}: C^{\infty}_0(M) \rightarrow C^{\infty}(\J_{int}(M,\omega))[[\nu]]:H \mapsto \left[ J\mapsto 4(2\pi)^m\nu^{m-1}\tr^{*_{J}}(H-\frac{\nu}{2} \Delta^J H)\right],$$ 
is a formal moment map on $(\J_{int}(M,\omega),\Omega^{\J_{int}})$.
\end{enumerate}

\noindent Moreover, at first order in $\nu$, both maps $\mu$ and $\widetilde{\mu}$ co\"incide with the Donaldson-Fujiki moment map.
\end{theoremprinc}

\noindent We will use the next two Lemmas.

\begin{lemma} \cite{gr3} \label{lemme:Liederiv}
Consider $H\in C^{\infty}(M)$, then the derivative of the action of $\varphi_t^{H}$ on $\Gamma\W\otimes\Lambda M$ is given by the formula:
\begin{equation*}
\ddt (\varphi_t^{H})^*=(\varphi_t^{H})^*\left(\imath(X_{H})D+D\imath(X_{H})+\invnu\left[- \omega_{ij}y^i X_{H}^j + \frac{1}{2}(\nabla^2_{kq}H) y^k y^q-\imath(X_{H})r,\cdot \right]\right),
\end{equation*}
where $D$ is the Fedosov flat connection obtained with symplectic connection $\nabla$ and a choice of a series of closed $2$-forms.
\end{lemma}

\begin{lemma}\label{lemme:QHformula}
Let $H\in C^{\infty}(M)$, $J\in \J(M,\omega)$ inducing the symplectic connection $\nabla^{J}$, we have:
\begin{eqnarray} \label{eq:QH}
Q^{J}(H) & = & H-\omega_{ij}y^i X_{H}^j + \frac{1}{2}((\nabla^J)^2_{kq}H) y^k y^q-\imath(X_{H})r^{J}+ \alpha_{J}(\Lr_{X_H}J) \\
& & - \nu(D_J)^{-1}\left(\imath(X_H)\rho^J + \frac{1}{2}\left(\delta^J\Lr_{X_H}J\right)^{b_J}\right). \nonumber
\end{eqnarray}
Moreover, if $J$ is integrable,
$$Q^{J}(H-\frac{\nu}{2}\Delta^J H) =  H-\frac{\nu}{2}\Delta^J H-\omega_{ij}y^i X_{H}^j + \frac{1}{2}((\nabla^J)^2_{kq}H) y^k y^q-\imath(X_{H})r^{J}+ \alpha_{J}(\Lr_{X_H}J)$$
\end{lemma}

\begin{proof}
In \cite{LLFformalDonaldson}, we obtained (adapted to the notations of the present paper)
\begin{eqnarray*}
Q^{J}(H)& = & H-\omega_{ij}y^i X_{H}^j + \frac{1}{2}((\nabla^J)^2_{kq}H) y^k y^q-\imath(X_{H})r^{J} \\
& & + (D^J)^{-1}\left(\ddto \overline{\Gamma}^{\varphi_{-t}^H.J}+\ddto r^{\varphi_{-t}^H.J}-\nu\imath(X_H)\rho^J\right),
\end{eqnarray*}
where the last term is what is hidden in the connection form in \cite{LLFformalDonaldson}.

In Equation \eqref{eq:QH}, we add what is needed to make appear $\alpha_{J}(\Lr_{X_H}J) $ provided that $\imath(X_H)\rho^J + \frac{1}{2}\left(\delta^J\Lr_{X_H}J\right)^{b_J}$ is a closed $1$-form. But that follows from the equivariance of the Hermitian Ricci form (Lemma \ref{lemme:equivHRicci})
$$d\imath(X_H)\rho^J=\ddto \rho^{\varphi_{-t}^H.J},$$
and from Corollary \ref{cor:variationHermitian}
$$d\left(\frac{1}{2}\left(\delta^J\Lr_{X_H}J \right)^{b_J}\right)=-\ddto \rho^{\varphi_{-t}^H.J}.$$

\noindent In the K\"ahler case, the formula follows from Lemma \ref{lemme:exactLaplacian} below.
\end{proof}

\begin{lemma} \label{lemme:exactLaplacian}
If $(M,\omega,J)$ is K\"ahler, then 
$$\imath(X_H)\rho^J + \frac{1}{2}\left(\delta^J\Lr_{X_H}J\right)^{b_J}=d(\frac{1}{2}\Delta^J H).$$
\end{lemma}

\noindent The proof of Lemma \ref{lemme:exactLaplacian} is postponed to the appendix.

\noindent We now prove the main Theorem.

\begin{proof}[Proof of Theorem \ref{theor:formalDFpicture}]
The proof is similar to the cases studied in \cite{LLFlast} and \cite{LLFformalDonaldson}. To shorten the proof we work directly with $\widetilde{\mu}$ whose expression makes sense in the almost-K\"ahler case and, at first order in $\nu$, co\"incides with $\mu$.

The equivariance is immediate from the naturality of the Fedosov construction and the equivariance of all of its ingredients from Proposition \ref{prop:nablaJ} and Lemma \ref{lemme:equivHRicci}. 

We check the formal moment map equation in the K\"ahler case. For $J\in\J(M,\omega)$, and the path $t\mapsto J_t\in \J(M,\omega)$ through $J$ such that $\ddto J_t =A \in T_J\J(M,\omega)$ 
$$ \left(d^{\J}\widetilde{\mu}(H)\right)(A)=4(2\pi)^m\nu^{m-1}\ddto\tr^{*_{J_t}}(H-\frac{\nu}{2}\Delta^J H)$$
Using Lemma \ref{lemme:tracevariation} and after the formulas from Lemmas \ref{lemme:Liederiv} and \ref{lemme:QHformula},
\begin{eqnarray*}
\ddto\tr^{*_{J_t}}(H-\frac{\nu}{2}\Delta^{J_t} H) & = & \tr^{*_{J}}\left(\left.\invnu[\alpha_{J}(A),Q^{J}(H-\frac{\nu}{2}\Delta^J H)]\right|_{y=0}\right)-\frac{\nu}{2}\tr^{*_{J}}(\ddto \Delta^{J_t} H),  \\
 & \hspace{-2.7cm} = & \hspace{-1.5cm} \tr^{*_{J}}\left(\left.\invnu[\alpha_{J}(A),\alpha_J(\Lr_{X_H}J)]\right|_{y=0}\right)+\tr^{*_{J}}\left(\left.-\ddto (\varphi_t^H)^*\alpha_{J}(A) \right|_{y=0}\right)\\
& &\hspace{-1.5cm} +\,\tr^{*_{J}}\left(\left.\left(\imath(X_{H})D^{J}+D^{J}\imath(X_{H})\right)\alpha_{J}(A)\right|_{y=0}\right) -\frac{\nu}{2}\tr^{*_{J}}(\ddto \Delta^{J_t} H)
\end{eqnarray*}
Now, $\alpha_{J}(A)$ is a $0$-form, all of its terms contain $y$'s. So, at $y=0$ it remains,
\begin{eqnarray*}
\ddto\tr^{*_{J_t}}(H-\frac{\nu}{2}\Delta^{J_t} H) & = & \tr^{*_{J}}\left(\left.\invnu[\alpha_{J}(A),\alpha_{J}(\Lr_{X_H}J)] + \imath(X_{H})D^{J}\alpha_{J}(A)\right|_{y=0}\right) \\
& & -\,\frac{\nu}{2}\tr^{*_{J}}(\ddto \Delta^{J_t} H).
\end{eqnarray*}
By the definition of $\alpha$ and $\Rr_J$, we have
\begin{eqnarray} \label{eq:endaKcase}
\ddto\tr^{*_{J_t}}(H-\frac{\nu}{2}\Delta^{J_t} H) & = & -\,\tr^{*_J}\left(\left.\Rr_J(\Lr_{X_H}J,A)\right|_{y=0}\right) -\frac{\nu}{4} \tr^{*_{J}}\left( \mathrm{Tr}(JA\Lr_{X_H}J)\right) \\
& & +\, \frac{\nu}{2} \tr^{*_J}\left((\delta^J A)^{b_J}(X_H)\right) - \frac{\nu}{2}\tr^{*_{J}}(\ddto \Delta^{J_t} H) \nonumber
\end{eqnarray}
Finally, in the K\"ahler case, the Lemma \ref{lemme:ddtoLaplacian} below implies
\begin{equation*} \label{eq:conclusion1}
 -\,\frac{\nu}{4} \tr^{*_{J}}\left( \mathrm{Tr}(JA\Lr_{X_H}J)\right) 
+ \frac{\nu}{2} \tr^{*_J}\left((\delta^J A)^{b_J}(X_H)\right) - \frac{\nu}{2}\tr^{*_{J}}(\ddto \Delta^{J_t} H)=0.
\end{equation*}
So that, one obtains the formal moment map equation 
$$\ddto 4(2\pi)^m\nu^{m-1} \tr^{*_{J_t}}(H-\frac{\nu}{2}\Delta^{J_t} H)  =-\left(\imath(\Lr_{X_H}J)\Omtilde\right)(A).$$

In the almost-K\"ahler case, Equation \eqref{eq:endaKcase} is still valid at order $1$ in $\nu$ as there is no contribution of the Laplacian because the trace starts with the integral functional. Hence, using Lemma \ref{lemme:formula}, one get
$$ \ddto 4(2\pi)^m\nu^{m-1} \tr^{*_{J_t}}(H)  =-\left(\imath(\Lr_{X_H}J)\Omtilde\right)(A)+O(\nu).$$

At first order in $\nu$, one knows (see \cite{FutLLF2} for example) the first terms of the normalised trace:
$$4(2\pi)^m\nu^{m-1}\tr^{*_{J}}(H)=-\int_M HS^J\dvol + O(\nu),$$
which is the Donaldson-Fujiki moment map.
\end{proof}

\begin{lemma} \label{lemme:ddtoLaplacian}
For $J\in \J_{int}(M,\omega)$ and $H\in C^{\infty}(M)$, one compute
$$\ddto \Delta^{J_t} H = (\delta^J A)^{b_J}(X_H) -  \mathrm{Tr}(JA\Lr_{X_H}J)$$
\end{lemma}

\begin{lemma} \label{lemme:formula}
For $J\in \J(M,\omega)$ and $H\in C^{\infty}(M)$, we have
$$\int_M (\delta^J A)^{b_J}(X_H) \dvol = \int_M \mathrm{Tr}(JA\Lr_{X_H}J)\dvol$$
\end{lemma}

\noindent The proofs of the above two Lemmas is contained in the Appendix.

\begin{rem}
We suspect that the Lemmas \ref{lemme:exactLaplacian} and \ref{lemme:ddtoLaplacian} are valid in the general almost-K\"ahler case. We postpone this task to a future work.
\end{rem}

Our last corollary, contains at first order in $\nu$ the link that was presented in \cite{LLF}. Namely, that closedness of $*_J$ translates into the vanishing of a (formal) moment map. We say $*_J$ is closed up to order $n$ if the integral is a trace for $*_J$ modulo terms in $\nu^{n+1}$.

\begin{cor} 
For $\J\in T_J\J_{int}(M,\omega)$.\\ The star product $*_J$ is closed up to order $n$ if and only if $\widetilde{\mu}(J)=0+O(\nu^{n+1})$.
\end{cor}

\begin{rem}
In the almost-K\"ahler case, one can state the same corollary involving $\mu$, but the (formal) moment map interpretation is only valid at order $1$ in $\nu$.
\end{rem}

\section*{Appendix}

In this appendix, we prove the key identities from Lemmas \ref{lemme:delta}, \ref{lemme:exactLaplacian}, \ref{lemme:ddtoLaplacian} and \ref{lemme:formula}

\begin{proof}[Proof of Lemma \ref{lemme:delta}]
We prove that 
$$\ddto\left[\delta^{J_{0t}}(B)\right]^{b_{J_{0t}}}- \ddto\left[\delta^{\widetilde{J}_{0t}}(A)\right]^{b_{\widetilde{J}_{0t}}}=-\frac{1}{2} d(\mathrm{Tr}(JAB)),$$
with the notations introduced in Theorem \ref{theor:Ralpha}.

To identify the above LHS, we compute for all $Y\in \VM$, the $L^2$-product of $\ddto\left[\delta^{\widetilde{J}_{0t}}(A)\right]^{b_{\widetilde{J}_{0t}}}$ with the $1$-form $g_J(Y,\cdot)$. With a frame $\{e_k\,|\,k=1,\ldots,2m\}$, we obtain
\begin{eqnarray*}
\int_M \ddto\left[\delta^{\widetilde{J}_{0t}}(A)\right]^{b_{\widetilde{J}_{0t}}}(e_k)g_J^{kl}g_J(Y,e_l)\dvol & \hspace{-0.2cm} = & \hspace{-0.2cm} \ddto \int_M \left[\delta^{\widetilde{J}_{0t}}(A)\right]^{b_{\widetilde{J}_{0t}}}(e_k)(g_{\widetilde{J}_{0t}})^{kl}g_{\widetilde{J}_{0t}}(Y,e_l) \\
 & \hspace{-0.2cm} = & \hspace{-0.2cm} \ddto  \int_M g_{\widetilde{J}_{0t}}(e_k, Ae_q))(g_{\widetilde{J}_{0t}})^{qp}(g_{\widetilde{J}_{0t}})^{kl}g_{\widetilde{J}_{0t}}(\nabla^{g_{\widetilde{J}_{0t}}}_{e_p} Y,e_l) \\
\end{eqnarray*}
Making use of Lemma \ref{lemme:firstvarnabla}, one obtains
\begin{eqnarray*}
\int_M \ddto\left[\delta^{\widetilde{J}_{0t}}(A)\right]^{b_{\widetilde{J}_{0t}}}(e_k)g^{kl}g(Y,e_l)\dvol & \hspace{-0.2cm} = & \hspace{-0.2cm} - \int_M g^{kl} g(\nabla^{g_J}_{e_k}Y,J(AB+BA)e_l)\dvol \\
 & & \hspace{-0.2cm} -\, \frac{1}{2}\int_M g^{qp} g^{kl} g(e_k,Ae_q)\left[\nabla^{g_J}_Yg(\cdot,JB\cdot)\right](e_k,e_l)\dvol.\nonumber
\end{eqnarray*}
Because the first term in the RHS above is symmetric in $A,B$, denoting by $(\cdot,\cdot)_J$ the $L^2$ product of tensors induced by $g_J$, we have for all $Y\in \VM$
\begin{eqnarray}
\left(\ddto\left[\delta^{J_{0t}}(B)\right]^{b_{J_{0t}}}- \ddto\left[\delta^{\widetilde{J}_{0t}}(A)\right]^{b_{\widetilde{J}_{0t}}},g(Y,\cdot) \right)_J &  = &  -\frac{1}{2} \left( g(\cdot,B\cdot),\left[\nabla^{g_J}_Y g(\cdot,JA\cdot)\right]\right)_J \nonumber\\
& & +\,\frac{1}{2}\left( g(\cdot,A\cdot),\left[\nabla^{g_J}_Yg(\cdot,JB\cdot)\right]\right)_J,\nonumber\\
 &  = & -\,\frac{1}{2}\left(d(\mathrm{Tr}(JAB)),g(Y,\cdot)\right)_J,\nonumber
\end{eqnarray}
which concludes the proof of the Lemma \ref{lemme:delta}.
\end{proof}

\begin{proof}[Proof of Lemma \ref{lemme:exactLaplacian}]
Considering $(M,\omega,J)$ is K\"ahler, we will show
$$\imath(X_H)\rho^J + \frac{1}{2}\left(\delta^J\Lr_{X_H}J\right)^{b_J}=d(\frac{1}{2}\Delta^J H).$$

\noindent In the K\"ahler case,
\begin{equation}\label{eq:LXHJ}
\Lr_{X_H}J(Y)=-\nablaLC_{JY}X_H+J\nablaLC_Y X_H.
\end{equation}
So that, denoting by $\beta=g(X_H,\cdot)$ and using an unitary frame $\{e_k\,|\,k=1,\ldots,2m\}$,
\begin{equation}\label{eq:deltaLXHJ}
\delta^J\left(\Lr_{X_H}J\right)^b = \sum_i (\nablaLC)^2_{(e_i,J e_i)} \beta - \delta^J\nablaLC (\beta \circ J).
\end{equation}
Now, since $\rho^J$ co\"incides with the Ricci form when $(M,\omega,J)$ is K\"ahler, Equation \eqref{eq:RicciKahler} leads to
\begin{eqnarray*}
\sum_i (\nablaLC)^2_{(e_i,J e_i)} \beta & = & \frac{1}{2} \sum_i R^{g_J}(e_i,Je_i)\beta \\
& = & -\imath(X_H)\rho^J 
\end{eqnarray*}
Using the Weitzenbock formula and $\beta\circ J = -dH$, the second term of Equation \eqref{eq:deltaLXHJ} becomes
\begin{eqnarray*}
 - \delta^J\nablaLC (\beta \circ J)& = & (\delta^J d + d\delta^J)dH + \sum_{i,j} e_i^*\wedge \imath(e_j) R(e_i,e_j)dH,\\
 & = & d(\Delta^J H) -\imath(X_H)\rho^J.
\end{eqnarray*}
So,
\begin{equation*}
\delta^J\left(\Lr_{X_H}J\right)^b =  -2\imath(X_H)\rho^J+ d(\Delta^J H).
\end{equation*}
which finishes the proof.
\end{proof}

\begin{proof}[Proof of Lemma \ref{lemme:ddtoLaplacian}]
In the K\"ahler setting, we will prove 
$$\ddto \Delta^{J_t} H = (\delta^J A)^{b_J}(X_H) -  \mathrm{Tr}(JA\Lr_{X_H}J).$$

\noindent The LHS is
$$\ddto \Delta^{J_t} H = \frac{1}{2}\ddto \Lambda^{ks}\left[d(-dH\circ J_t)\right]_{ks}$$
So that in a frame $\{e_k\, |\,k=1,\ldots,2m\}$, we have
\begin{eqnarray*}
\ddto \Delta^{J_t} H & = & \frac{1}{2}\ddto \Lambda^{ks}\left[d(-dH\circ A)\right]_{ks} \\
 & = &-\Lambda^{ks}\left( e_k(\omega(X_H,Ae_s))-\omega(X_H,A\nablaLC_{e_k} e_s)\right) \\
 & = & -\Lambda^{ks}\left( \omega(\nablaLC_{e_k}X_H,Ae_s))-\omega(X_H,(\nablaLC_{e_k}A) e_s)\right)
\end{eqnarray*}
By Equation \eqref{eq:LXHJ}, we get
$$\ddto \Delta^{J_t} H =  -\, \frac{1}{2} \mathrm{Tr}(JA\Lr_{X_H}J)+ (\delta^J A)^{b_J}(X_H).$$
which concludes the proof.
\end{proof}

\begin{proof}[Proof of Lemma \ref{lemme:formula}]
Let us prove finally that 
$$\int_M (\delta^J A)^{b_J}(X_H) \dvol = \int_M \mathrm{Tr}(JA\Lr_{X_H}J)\dvol$$
Using the notation $(\cdot,\cdot)_J$ for the $L^2$-product of tensors, then :
\begin{eqnarray*}
\int_M \mathrm{Tr}(JA\Lr_{X_H}J)\dvol & = &  (g_J(JA\cdot,\cdot),g_J(\Lr_{X_H}J\cdot,\cdot))_J  
\end{eqnarray*}
From $\Lr_{X_H}\omega=0$, we obtain
$$g_J((\Lr_{X_H}J)U,V)=-g_J(\nablaLC_{JU}X_H,V) - g_J(JU, \nablaLC_V X_H).$$
So that,
$$(g_J(JA\cdot,\cdot),g_J(\Lr_{X_H}J\cdot,\cdot))_J= \int_M (\delta^J A)^{b_J}(X_H) \dvol.$$
The proof is over.
\end{proof}

\end{document}